\documentclass[french,11pt,titlepage,thmsa,a4paper]{amsart}

\usepackage{graphicx}
\usepackage[applemac]{inputenc}
\usepackage{amsmath,amssymb}
\usepackage[frenchb]{babel}
\usepackage{mathrsfs}
\usepackage[a4paper]{geometry}
\newtheorem{theorem}{Théorème}

\newtheorem{corollary}[theorem]{Corollaire}

\newtheorem{definition}[theorem]{Définition}
\newtheorem{example}[theorem]{Exemple}

\newtheorem{lemma}[theorem]{Lemme}

\newtheorem{proposition}[theorem]{Proposition}
\newtheorem{remark}[theorem]{Remarque}

\begin{document}

\begin{center}
\textsf{\LARGE{}Critère d'existence d'idempotent basé sur les algèbres
de Rétrocroisement}
\par\end{center}{\LARGE \par}

\begin{center}

\par\end{center}

\begin{center}
\textbf{Cristi\'{a}n Mallol} $^{a}$, \textbf{Richard Varro} $^{b}$\bigskip{}
 
\par\end{center}

$^{a}$ {\footnotesize{}Departamento de Ingener\'{i}a Matem\'{a}tica, Universidad
de La Frontera, Casilla 54-D, Temuco, Chile}{\small \par}

$^{b}$ {\footnotesize{}Institut de Mathématiques et de Modélisation de Montpellier,
Université de Montpellier II, 35095 Montpellier Cedex 5, France.}{\small \par}

$^{b}${\footnotesize{} Département de Mathématiques et Informatique Appliquées,
Université de Montpellier III, 34199 Montpellier Cedex 5, France}{\small \par}

\bigskip{}

\textbf{\small{}Abstract}{\small{} : We study the relationship of
backcrossing algebras with mutation algebras and algebras satisfying
$\omega$-polynomial identities: we show that in a backcrossing algebra
every element of weight 1 generates a mutation algebra and that for
any polynomial identity $f$ there is a backcrossing algebra satisfying
$f$. We give a criterion for the existence of idempotent in the case
of baric algebras satisfying a nonhomogeneous polynomial identity
and containing a backcrossing subalgebra. We give numerous genetic
interpretations of the algebraic results. }{\small \par}

\medskip{}

\emph{\small{}Key words}{\small{} : Backcrossing algebras, baric algebras,
mutation algebras, idempotent element, $\omega$-polynomial identities.}{\small \par}

\emph{\small{}2008 MSC}{\small{} : Primary : 17D92, Secondary : 92A10.}{\small \par}

\medskip{}

\section{Introduction}

Dans ce qui suit $K$ est un corps commutatif de caractéristique différente
de $2$ et $\left(A,\omega\right)$ est une $K$-algèbre commutative
pondérée, c'est-à-dire qu'il existe un morphisme d'algèbres non nul
$\omega:A\rightarrow K$. Pour $x\in A$, le scalaire $\omega\left(x\right)$
est appelé le poids de $x$, on note $H_{\omega}=\left\{ x\in A:\omega\left(x\right)=1\right\} $
l'hyperplan affine des éléments de $A$ qui sont de poids 1.

Soit $K\left\langle X\right\rangle $ (resp. $K\left[X\right]$) la
$K$-algèbre libre commutative non-associative (resp. associative)
des polynômes d'indéterminée $X$.

\'{E}tant donnés $\left(A,\omega\right)$ une $K$-algèbre pondérée
et $f=\sum_{i=1}^{m}f_{i}$ un polynôme de $K\left\langle X\right\rangle $
où $f_{i}\in K\left\langle X\right\rangle $ tels que $\mbox{deg}f_{i}\geq\mbox{deg\ensuremath{f_{i-1}}}$,
on dit que $f$ est homogène si $\mbox{deg}f_{i}=\mbox{deg\ensuremath{f}}$
pour tout $1\leq i\leq m$, sinon $f$ est dit non homogène. Et on
dit que l'algèbre $A$ vérifie l'identité $f$ si on a $\sum_{i=1}^{m}\omega\left(x\right)^{\text{deg }f-\text{ deg }f_{i}}f_{i}\left(x\right)=0$
pour tout $x\in A$. Dans le cas où $\mathrm{card}K>\mathrm{deg}f$,
l'algèbre $A$ vérifie l'identité $f\in K\left\langle X\right\rangle $
si et seulement si on a $f\left(y\right)=0$ pour tout $y\in H_{\omega}$.
Une algèbre vérifiant une identité $f\in K\left\langle X\right\rangle $
est dite $\omega$-polynomiale.
\begin{definition}
Une $K$-algèbre pondérée est de rétrocroisement si elle vérifie l'identité
$X^{2}X^{2}-2X^{3}+X^{2}$ et ne vérifie pas d'identité de degré $<4$. 
\end{definition}

\textbf{Interprétation génétique}. Le nom de cette algèbre provient
de son interprétation génétique. Dans une population, la loi de transmission
d'un type génétique s'interprète algébriquement par la donnée d'une
algèbre non associative pondérée $A$ (\cite{Lyub}, \cite{WB}).
Dans cette algèbre, à tout élément $x$ de poids $1$, correspond
une structure génétique de la population, c'est-à-dire une distribution
de fréquences. De plus si $x$ est la structure génétique de la génération
parentale $P$, alors le calcul de $x^{2}$ et $x^{2}x^{2}$ fournit
les structures génétiques respectivement de la première et de la seconde
génération issues de $P$, notées $F_{1}$ et $F_{2}$ (cf. \cite{WB},
p. 7). De manière analogue le calcul de $x^{3}$ fournit la structure
génétique de la population issue du croisement des générations $P$
et $F_{1}$, appelé le rétrocroisement (ou ``backcrossing''). Compte
tenu de ce qui précède, la relation 
\[
x^{3}=\frac{1}{2}\left(x^{2}x^{2}+x^{2}\right)
\]
 signifie que la structure génétique de la population issue du rétrocroisement
est la moyenne arithmétique des distributions génétiques des générations
$F_{1}$ et $F_{2}$.\medskip{}

Comme on l'a remarqué dans \cite{M-V-05}, cette algèbre est ubiquiste,
elle apparaît dans plusieurs situations : 
\begin{itemize}
\item comme identité vérifiée par toute algèbre de mutation (\cite{M-V-94},
remarque 2.8); 
\item en tant qu'algèbre commune aux deux familles issues de l'études des
algèbres $\omega$-monomiales de monôme directeur $x^{2}x^{2}$ (\cite{M-S-99},
remarque 3.6); 
\item comme identité invariante universelle par gamétisation des relations
de monôme directeur $x^{2}x^{2}$ (\cite{M-V-03}, prop. 22); 
\item comme relation assurant l'existence d'un idempotent dans les train
algèbres plénières de rang 4 (\cite{LabSua-07}) et de rang $n$ (\cite{B-H-10});
\item comme critère intervenant dans la partition de la classe des algèbres
qui vérifient l'identité $X^{2}X^{2}-X^{4}-\alpha X^{3}-\beta X^{2}+\left(\alpha+\beta\right)X$
(\cite{N-V-11,N-V-12}).\end{itemize}
\begin{proposition}
\label{prop:UnicPond}La pondération d'une algèbre de rétrocroisement
est unique.\end{proposition}
\begin{proof}
Soient $\left(A,\omega\right)$ une algèbre de rétrocroisement et
$\eta$ une pondération de $A$. Alors pour tout $x\in A$, en appliquant
$\eta$ à $\left(x^{2}-\omega\left(x\right)x\right)^{2}=0$ on a $\eta\left(x\right)\left(\eta\left(x\right)-\omega\left(x\right)\right)=0$,
il en résulte que $H_{\eta}\subset H_{\omega}$ et $\ker\omega\subset\ker\eta$.
Soient $e\in H_{\eta}$ et $z\in\ker\eta$ on a $e+z\in H_{\eta}$
d'où $\omega\left(e\right)=1$ et $\omega\left(e+z\right)=1$, on
en déduit que $z\in\ker\omega$ et donc $\ker\eta\subset\ker\omega$.
Soit $x\in H_{\omega}$, on a $x^{2}-x\in\ker\omega=\ker\eta$ d'où
$\eta\left(x\right)\left(\eta\left(x\right)-1\right)=0$ on a $\eta\left(x\right)\neq0$
car $x\notin\ker\omega$, par conséquent $x\in H_{\eta}$ et donc
$H_{\omega}=H_{\eta}$. Il résulte de tout ceci que $\eta=\omega$.
\end{proof}

\section{Algèbres de rétrocroisement et algèbres de mutation}

Les algèbres de mutation et de rétrocroisement sont très liées. En
effet, d'après \cite{M-V-94} (remarque 2.8) toute algèbre de mutation
est de rétrocroisement. On va voir que toute algèbre de rétrocroisement
contient des algèbres de mutation.

On fera appel plusieurs fois au lemme suivant. 
\begin{lemma}
\label{baseK<x>}Soit $\left(A,\omega\right)$ une algèbre de rétrocroisement,
pour $x\in H_{\omega}$ et tout $k\geq1$ on pose $p_{k}=x^{k}-x^{k-1}$
où par convention $x^{0}=0$. La famille $\left(p_{k}\right)_{k\geq1}$
est une famille génératrice de $K\left\langle x\right\rangle $ telle
que $p_{1}^{2}=p_{1}+p_{2}$, $p_{1}p_{i}=p_{i+1}$ et $p_{i}p_{j}=0$
pour tout $i,j\geq2$. \end{lemma}
\begin{proof}
Soient $\left(A,\omega\right)$ une algèbre de rétrocroisement et
$x\in H_{\omega}$. La famille $\left(p_{k}\right)_{k\geq1}$ est
génératrice de $K\left\langle x\right\rangle $, car pour tout $k\geq1$
on a $x^{k}=\sum_{i=1}^{k}p_{i}$. Montrons que l'espace engendré
par la famille $\left(p_{k}\right)_{k\geq1}$ est une sous-algèbre
de $A$. Il est clair que $p_{1}^{2}=p_{1}+p_{2}$, $p_{1}p_{i}=p_{k+1}$
pour $i\geq2$. Montrons par récurrence sur l'entier $r\geq4$, que
l'on a $p_{i}p_{j}=0$ pour tout $i,j\geq2$ tels que $i+j=r.$

Ce résultat est vrai pour $r=4$ car l'identité de rétrocroisement
se traduit par $p_{2}^{2}=0$. Supposons que $p_{i}p_{j}=0$, pour
tout $i,j\geq2$ tels que $4\leq i+j\leq r$. Par linéarisation de
l'identité $\left(x^{2}-\omega\left(x\right)x\right)^{2}=0$ on obtient
la relation : 
\begin{equation}
R\left(x,y,z,t\right):\quad G\left(x,y\right)G\left(z,t\right)+G\left(x,z\right)G\left(y,t\right)+G\left(x,t\right)G\left(y,z\right)=0\label{LinRetro}
\end{equation}
 où $x$, $y$, $z$, $t\in A$ et $G\left(a,b\right)=2ab-\omega\left(a\right)b-\omega\left(b\right)a$.

Alors la relation $R\left(p_{1},p_{1},p_{k-1},p_{r-k}\right)$ où
$2\leq k\leq r-1$ s'écrit 
\[
4p_{2}\left(p_{k-1}p_{r-k}\right)+2\left(2p_{k}-p_{k-1}\right)\left(2p_{r-k+1}-p_{r-k}\right)=0
\]
 ce qui se simplifie avec l'hypothèse de récurrence en $p_{k}p_{r-k+1}=0$.
\end{proof}
\medskip{}

On rappelle qu'une algèbre de mutation \cite{M-V-94} est la donnée
d'un triplet $\left(V,M,\eta\right)$ où $V$ est un $K$-espace vectoriel,
$M$ un endomorphisme de $V$ et $\eta$ une forme linéaire non nulle
sur $V$ tels que $\eta\circ M=\eta$ et $V$ muni de la structure
d'algèbre $xy=\frac{1}{2}\left[\eta\left(y\right)M(\left(x\right)+\eta\left(x\right)M\left(y\right)\right]$.
On en déduit sans peine que $\eta$ est une pondération et que $M\neq0$.\medskip{}

\textbf{Interprétation génétique}. Les algèbres de mutation apparaîssent
dans \cite{Lyub80}, on a montré dans \cite{M-V-00} qu'elles modèlisent
des populations subdivisées en plusieurs colonies entre lesquelles
ont lieu des migrations dont les taux dépendent des types génétiques
et à l'intérieur desquelles se produisent des mutations génétiques
avec des taux qui dépendent des colonies. 
\begin{theorem}
\label{Mut&poids1}Dans une algèbre de rétrocroisement $\left(A,\omega\right)$,
tout élément de poids 1 engendre une algèbre de mutation pondérée
par $\omega$. \end{theorem}
\begin{proof}
Soient $\left(A,\omega\right)$ une algèbre de rétrocroisement et
$x\in H_{\omega}$. D'après le lemme \ref{baseK<x>}, le sous-espace
engendré par $\left\{ p_{k}\right\} _{k\geq2}$ où $p_{k}=x^{k}-x^{k-1}$
est un idéal nilpotent d'indice 2 de la la sous-algèbre $K\left\langle x\right\rangle $
engendrée par $x$, donc d'après \cite{M-V-94}, $K\left\langle x\right\rangle $
est de mutation. De plus la structure d'algèbre de mutation de $K\left\langle x\right\rangle $
est $\left(K\left\langle x\right\rangle ,M,\omega\right)$ où $\omega$
est la pondération de $A$ et $M$ l'endomorphisme de $K\left\langle x\right\rangle $
défini selon la dimension de $K\left\langle x\right\rangle $ par
$M\left(p_{1}\right)=p_{2}+p_{1}$, $M\left(p_{k}\right)=2p_{k+1}$
pour $k\geq2$ si $K\left\langle x\right\rangle $ est de dimension
infinie et $M\left(p_{k}\right)=2p_{k+1}$ pour $2\leq k\leq n-1$,
$M\left(p_{n}\right)=2\sum_{i=2}^{n}\alpha_{i}p_{i}$ si $K\left\langle x\right\rangle $
est de dimension finie $n$. 
\end{proof}
Dans une algèbre commutative $A$ on définit usuellement deux types
de puissances pour un élément $x\in A$ : les puissances principales
définies par $x^{1}=x$ et $x^{k+1}=xx^{k}$ et les puissances plénières
définies par $x^{\left[1\right]}=x$ et $x^{\left[k+1\right]}=x^{\left[k\right]}x^{\left[k\right]}$
pour tout entier $k\geq1$. Dans les algèbres de rétrocroisement il
existe des relations entre ces puissances. 
\begin{corollary}
Dans une algèbre de rétrocroisement $\left(A,\omega\right)$, pour
tout $x\in H_{\omega}$ et tout entier $k\geq1$ on a : 
\[
x^{k+1}=\frac{1}{2^{k-1}}x^{\left[k+1\right]}+\sum_{i=1}^{k-1}\frac{1}{2^{i}}x^{\left[i+1\right]}\quad et\quad x^{\left[k+1\right]}=2^{k-1}x^{k+1}-\sum_{i=2}^{k}2^{i-2}x^{i}.
\]
 \end{corollary}
\begin{proof}
Soit $x\in H_{\omega}$, d'après le théorème \ref{Mut&poids1} la
sous-algèbre $K\left\langle x\right\rangle $ est de mutation, il
existe donc un endomorphisme $M$ de $K\left\langle x\right\rangle $
tel que pour tout $z\in K\left\langle x\right\rangle $ on a $z^{2}=\omega\left(z\right)M\left(z\right)$.
Or, pour tout $k\geq1$ on vérifie aisément que $M^{k}\left(x\right)=2^{k-1}x^{k+1}-\sum_{i=2}^{k}2^{i-2}x^{i}$,
$x^{k+1}=\left[\frac{1}{2^{k-1}}M^{k}+\sum_{i=1}^{k-1}\frac{1}{2^{i}}M^{i}\right]\left(x\right)$
et $x^{\left[k+1\right]}=M^{k}\left(x\right)$ d'où le résultat. 
\end{proof}

\textbf{Interprétation génétique}. Génétiquement ces relations signifient
que dans une population vérifiant la relation de rétro-croisement,
si $x$ est la distribution des fréquences des gènes dans la génération
parentale, la structure génétique du $k$-ième rétrocroisement donnée
par $x^{k+1}$ est une moyenne des distributions des générations $F_{1}$,
\ldots{} , $F_{k}$ et la distribution de la génération $F_{k}$
est connue si l'on connaît celles des distributions des $k$ premiers
rétrocroisements.

\section{Algèbres de rétrocroisement et algèbres $\omega$-polynomiales}

Comme on l'a remarqué dans \cite{M-V-05} les algèbres de rétrocroisement
sont ubiquistes. En effet, comme le montre le résultat qui suit, on
en trouve dans toute algèbre pondérée vérifiant une identité polynomiale. 
\begin{proposition}
Pour tout $f\in K\left\langle X\right\rangle $ tel que $f=\sum_{i=1}^{m}\alpha_{i}f_{i}$
avec $m\geq2$ et $\sum_{i=1}^{m}\alpha_{i}=0$, il existe une algèbre
de rétrocroisement qui vérifie $f$. \end{proposition}
\begin{proof}
Dans \cite{N-V-11a} on a montré que pour tout polynôme non homogène
$f\in K\left\langle X\right\rangle $ il existe une algèbre de mutation
qui vérifie $f$. Ce résultat subsiste si $f$ est homogène, soit
$f=\sum_{i=1}^{m}\alpha_{i}f_{i}\in K\left\langle X\right\rangle $
un polynôme homogène tel que $\sum_{i=1}^{m}\alpha_{i}=0$. Comme
$f$ est homogène et $m\geq2$ on a $\mbox{deg}f\geq4$. Si $\left(V,M,\omega\right)$
est une algèbre de mutation, pour $x\in H_{\omega}$ et tout entier
$p,q\geq1$ on a 
\begin{eqnarray}
x^{\left[p+1\right]}=M^{p}\left(x\right), &  & x^{p+1}=\left[\frac{1}{2^{p-1}}M^{p}+\sum_{k=1}^{p-1}\frac{1}{2^{k}}M^{k}\right]\left(x\right),\label{eq:puissances}\\
M^{p}\left(x\right)M^{q}\left(x\right) & = & \frac{1}{2}\left[M^{p+1}\left(x\right)+M^{q+1}\left(x\right)\right],
\end{eqnarray}
en utilisant ces relations, il existe $\beta_{1},\ldots,\beta_{n}\in K$
tels que $f\left(x\right)=\sum_{k=1}^{n}\beta_{k}M^{k}\left(x\right)$
où $n\leq m$. On a deux cas. 

-- Si $\beta_{k}=0$ pour tout $1\leq k\leq n$ (voir par exemple
$f\left(X\right)=X^{3}X^{3}-\left(X^{2}X^{2}\right)X^{2}$), dans
ce cas on considère l'espace $V=K^{n}$, les endomorphismes $M$ et
$\eta$ définis sur la base canonique $\left(v_{0},\ldots,v_{n-1}\right)$
de $V$ par $M\left(v_{i}\right)=v_{i+1}$ si $i<n-1$, $M\left(v_{n-1}\right)=v_{0}$
et $\eta\left(v_{i}\right)=1$ pour tout $0\leq i<n$, il est clair
que $\left(V,M,\eta\right)$ est une algèbre de mutation vérifiant
l'identité $f$. 

-- Si $\beta_{1},\ldots,\beta_{n}\in K$ ne sont pas tous nuls, de
$\mbox{deg}f\geq4$ il résulte $n\geq2$ et $\beta_{n}\neq0$, de
$\sum_{i=1}^{m}\alpha_{i}=0$ il découle $\sum_{k=1}^{n}\beta_{k}=0$.
Si l'on pose $P\left(X\right)=\sum_{k=1}^{n}\beta_{n}^{-1}\beta_{k}X^{k}$
on peut écrire l'identité $f$ sous la forme $f\left(x\right)=\beta_{n}P\left(M\right)\left(x\right)$.
Alors, en prenant $V=K^{n}$ et $M$ définie sur la base canonique
$\left(v_{0},\ldots,v_{n-1}\right)$ de $V$ par la matrice compagnon
du polynôme $P$, on munit $V$ de la pondération $\eta\left(v_{n-1}\right)=1$
et $\ker\eta=K\left\langle v_{0},\ldots,v_{n-2}\right\rangle $, alors
de $M\left(v_{k}\right)=v_{k+1}$ pour $0\leq k\leq n-1$ et $M\left(v_{n-1}\right)=\sum_{k=1}^{n-2}\beta_{n}^{-1}\beta_{k}v_{k}$
on a $\eta\circ M=\eta$. Avec tout ceci on montre sans peine que
l'algèbre de mutation $\left(V,M,\eta\right)$ vérifie $f$. 

Dans les deux cas il existe une algèbre de mutation vérifiant l'identité
$f$, or une algèbre de mutation est de rétrocroisement.\end{proof}
\begin{proposition}
Soient $\left(A,\omega\right)$ une algèbre de rétrocroisement et
$x\in H_{\omega}$, alors pour tout $f$, $g\in K\left\langle X\right\rangle $
tels que $f\left(x\right)$, $g\left(x\right)\in\ker\omega$ on a
$f\left(x\right)g\left(x\right)=0$. \end{proposition}
\begin{proof}
Soit $x\in H_{\omega}$, pour tout $f$, $g\in K\left\langle X\right\rangle $
on a $f\left(x\right)$, $g\left(x\right)\in K\left\langle x\right\rangle $,
alors en utilisant la structure d'algèbre de mutation de $K\left\langle x\right\rangle $,
si on a $f\left(x\right)$, $g\left(x\right)\in\ker\omega$ on a aussitôt
que $f\left(x\right)g\left(x\right)=0$. 
\end{proof}

\textbf{Interprétation génétique}. En particulier, pour $f\left(X\right)=g\left(X\right)=X^{\left[i\right]}-X^{\left[j\right]}$
où $i$, $j\geq1$ on a : 
\[
\left(x^{\left[i\right]}-x^{\left[j\right]}\right)^{2}=0,\forall x\in H_{\omega}.
\]
Ce résultat a une interprétation génétique surprenante : si dans une
population la structure génétique de la génération parentale $P$
(notée aussi $F_{0}$) vérifie la relation de rétrocroisement, c'est-à-dire
que la distribution des fréquences des gènes dans la population obtenue
par rétrocroisement de la $F_{1}$ avec $P$ est la moyenne des distributions
de fréquences des générations $F_{1}$ et $F_{2}$, alors cette relation
est vérifiée quelles que soient les générations $F_{i-1}$ et $F_{j-1}$
considérées. Formellement on a $F_{i-1}\times F_{j-1}=\frac{1}{2}\left(F_{i}+F_{j}\right)$
pour tout $i,j\geq1$.\medskip{}

De plus dans toute algèbre de rétrocroisement on a : 
\begin{theorem}
Soient $\left(A,\omega\right)$ une algèbre de rétrocroisement de
dimension finie et $S$ une partie au plus dénombrable de $K\left\langle X\right\rangle $
telle que $\mathrm{\mbox{card}}\left(S\right)\geq\dim A$. Alors pour
tout $x\in H_{\omega}$ il existe $f\in Lin\left(S\right)$ tel que
l'algèbre $K\left\langle x\right\rangle $ vérifie $f$.\end{theorem}
\begin{proof}
Soient $\left(A,\omega\right)$ une algèbre de rétrocroisement de
dimension finie et $S=\left\{ f_{k}\in K\left\langle X\right\rangle ;k\geq1\right\} $.
Soit $x\in H_{\omega}$, d'après le théorème \ref{Mut&poids1} la
sous-algèbre $K\left\langle x\right\rangle $ est de mutation, il
existe donc un endomorphisme $M$ de $K\left\langle x\right\rangle $
tel que pour tout $z\in K\left\langle x\right\rangle $ on a $z^{2}=\omega\left(z\right)M\left(z\right)$.
En utilisant les relations (\ref{eq:puissances}) on peut associer
à chaque $f_{k}$ un polynôme $\mu_{k}\in K\left[X\right]$ tel que
$f_{k}\left(y\right)=\mu_{k}\left(M\right)\left(y\right)$ pour tout
$y\in K\left\langle x\right\rangle \cap H_{\omega}$. Soit $\mu_{M}$
le polynôme minimal de $M$ et notons $\widetilde{\mu}_{k}$ le reste
de $\mu_{k}$ modulo $\mu_{M}$. On a trois possibilités :

-- s'il existe $k\geq1$ tel que $\widetilde{\mu}_{k}=0$ on a $f_{k}\left(y\right)=\mu_{k}\left(My\right)=\widetilde{\mu}_{k}\left(My\right)=0$
pour tout $y\in K\left\langle x\right\rangle $, donc $K\left\langle x\right\rangle $
vérifie l'identité $f_{k}$. 

Si on a $\widetilde{\mu}_{k}\neq0$ quel que soit $k\geq1$, on a
deux cas :

-- ou bien il existe $k\neq l$ tel que $\widetilde{\mu}_{k}=\widetilde{\mu}_{l}$
alors $\left(f_{k}-f_{l}\right)\left(y\right)=\left(\mu_{k}-\mu_{l}\right)\left(My\right)=\left(\widetilde{\mu}_{k}-\widetilde{\mu}_{l}\right)\left(My\right)=0$
et dans ce cas $K\left\langle x\right\rangle $ vérifie l'identité
$f_{k}-f_{l}$.

-- ou bien pour tout $k\neq l$ on a $\widetilde{\mu}_{k}\neq\widetilde{\mu}_{l}$
et dans ce cas la famille $\left\{ \widetilde{\mu}_{k};k\geq1\right\} $
est lié. En effet pour tout $k\geq1$ on a $\mbox{deg}\widetilde{\mu}_{k}<\mbox{deg}\mu_{M}\leq\dim A$,
donc le sous-espace engendré par $\left\{ \widetilde{\mu}_{k};k\geq1\right\} $
est de dimension finie $\leq\dim A$, mais l'ensemble $\left\{ \widetilde{\mu}_{k};k\geq1\right\} $
ayant pour cardinal $\mathrm{\mbox{card}}\left(S\right)\geq\dim A$
on a donc $\left\{ \widetilde{\mu}_{k};k\geq1\right\} $ lié. Il existe
donc des entiers $\left(k_{i}\right)_{1\leq i\leq m}$ et $\left(\lambda_{k_{i}}\right)_{1\leq i\leq m}\in K^{m}$
tels que $\lambda_{k_{i}}\neq0$ pour $1\leq i\leq m$ et $\sum_{i=1}^{m}\lambda_{k_{i}}\widetilde{\mu}_{k_{i}}=0$.
Alors pour tout $y\in K\left\langle x\right\rangle $ on a 
\[
\sum_{i=1}^{m}\lambda_{k_{i}}f_{k_{i}}\left(y\right)=\sum_{i=1}^{m}\lambda_{k_{i}}\mu_{k_{i}}\left(My\right)=\left(\sum_{i=1}^{m}\lambda_{k_{i}}\widetilde{\mu}_{k_{i}}\right)\left(My\right)=0
\]
 c'est-à-dire que la sous-algèbre $K\left\langle x\right\rangle $
vérifie l'identité $\sum_{i=1}^{m}\lambda_{k_{i}}f_{k_{i}}$. 
\end{proof}
\textbf{Interprétation génétique}. En particulier pour $S=\left\{ X^{\left[k\right]};k\geq1\right\} $,
il découle de ce résultat que dans une population dont un type génétique
est associé à une algèbre de rétrocroisement, il existe une génération
$F_{n}$ à partir de laquelle les structures génétiques des générations
$F_{k}$ ($k>n$) sont moyennes des distributions des générations
$F_{0},\ldots,F_{n}$.

\section{Rétrocroisement et existence d'un idempotent pour les algèbres $\omega$-polynomiales}

La notion d'idempotent joue un rôle important dans l'étude des algèbres
non associatives. Dans les algèbres de rétrocroisement l'existence
d'un idempotent n'est pas certaine, c'est-à-dire qu'il en existe avec
et d'autre sans, en effet dans \cite{M-V-02} on a des exemples d'algèbres
de mutation qui ont cette propriété. Néanmoins il est facile de voir
qu'une algèbre de rétrocroisement $\left(A,\omega\right)$ admet un
idempotent si et seulement si il existe $x\in H_{\omega}$ tel que
$x^{3}=x^{2}$.

Dans \cite{B-H-10} et \cite{LabSua-07}, les auteurs donnent des
conditions d'existence d'un idempotent pour les algèbres qui sont
à la fois train plénières et de rétrocroisement. Dans ce qui suit
nous allons étendre ces résultats aux algèbres vérifiant des identités
polynomiales non homogènes, pour cela nous aurons besoin de la notion
suivante introduite dans \cite{N-V-11a}. A chaque identité non homogène
$f\in K\left\langle X\right\rangle $ vérifiée par une algèbre $\left(A,\omega\right)$
on associe un polynôme $\vartheta_{f}\in K\left[X\right]$. On a donné
deux constructions de ce polynôme, dans ce qui suit on décrit celle
utilisant les algèbres de mutation. Etant donné $f\in K\left\langle X\right\rangle $,
il existe une algèbre de mutation $\left(V,M,\eta\right)$ qui vérifie
l'identité $f$ et sur $H_{\eta}$ l'identité $f\left(x\right)=0$
s'écrit $D\left(M\right)\left(x\right)=0$ avec $D\in K\left[X\right]$
et on pose $\vartheta_{f}\left(X\right)=D\left(2X\right)$. On a conjecturé
que si $\vartheta_{f}^{\prime}\left(\frac{1}{2}\right)\neq0$ alors
l'algèbre $A$ a au moins un idempotent. 
\begin{example}
\label{PolConjTrainPrinc}Une algèbre $\left(A,\omega\right)$ est
train principale de degré $n\geq2$ si elle vérifie une identité du
type $f\left(X\right)=X^{n}-\sum_{k=1}^{n-1}\alpha_{k}X^{k}$, en
prenant $T\in K\left[X\right]$ tel que $f\left(X\right)=XT\left(X\right)$,
ceci est équivalent à $T\left(L_{x}\right)\left(x\right)=0$ pour
tout $x\in H_{\omega}$, où $L_{x}:A\rightarrow A$, $y\mapsto xy$,
ce polynôme $T$ est appelé le train polynôme principal de $A$.

Soit $\left(A,\omega\right)$ une train algèbre de degré $n+1$ vérifiant
pour tout $x\in H_{\omega}$ : $f\left(x\right)=x^{n+1}-\sum_{k=1}^{n}\alpha_{k}x^{k}=0$
avec $\sum_{k=1}^{n}\alpha_{k}=1$. Dans une algèbre de mutation $\left(V,M,\eta\right)$
on a $x^{k}=\left(\frac{1}{2^{k-2}}M^{k-1}+\sum_{i=1}^{k-2}\frac{1}{2^{i}}M^{i}\right)\left(x\right)$
pour tout $x\in H_{\eta}$ et tout entier $k\geq2$, avec ceci l'identité
$f\left(x\right)=0$ s'écrit $ $
\[
\left[\frac{1}{2^{n-1}}M^{n}+\sum_{i=1}^{n-1}\frac{1}{2^{i}}M^{i}-\sum_{k=2}^{n-1}\alpha_{k}\left(\frac{1}{2^{k-2}}M^{k-1}+\sum_{i=1}^{k-2}\frac{1}{2^{i}}M^{i}\right)-\alpha_{1}id\right]\left(x\right)=0,
\]
d'où
\[
\vartheta_{f}\left(X\right)=2X^{n}+\sum_{i=1}^{n-1}X^{i}-\sum_{k=2}^{n-1}\alpha_{k}\left(2X^{k-1}+\sum_{i=1}^{k-2}X^{i}\right)-\alpha_{1}.
\]
Et en utilisant $2p\left(\frac{1}{2}\right)^{p-1}+\sum_{i=1}^{p-1}i\left(\frac{1}{2}\right)^{i-1}=4-4\left(\frac{1}{2}\right)^{p}$
et $\sum_{k=1}^{n}\alpha_{k}=1$, on obtient $\vartheta'_{f}\left(\frac{1}{2}\right)=-4T\left(\frac{1}{2}\right)$
où $T$ est le polynôme défini par $f\left(X\right)=XT\left(X\right)$
et appelé le train polynôme principal de $A$.\end{example}
\begin{proposition}
Soit $\left(A,\omega\right)$ une train algèbre principale de degré
$n+1\geq2$. S'il existe $x\in H_{\omega}$ tel que la sous-algèbre
$K\left\langle x\right\rangle $ soit de rétrocroisement de dimension
$n$ et si $\frac{1}{2}$ n'est pas train racine de $A$, alors $A$
possède un idempotent. \end{proposition}
\begin{proof}
Soit $\left(A,\omega\right)$ une train algèbre de polynôme principal
$T\in K\left[X\right]$ tel que $\deg\left(T\right)=n\geq1$ et $T\left(\frac{1}{2}\right)\neq0$.
On suppose qu'il existe $x\in H_{\omega}$ tel que la sous-algèbre
$K\left\langle x\right\rangle $ soit de rétrocroisement de dimension
$n$. Le cas $n=1$ étant trivial, on suppose désormais que $n\geq2$.
L'algèbre $K\left\langle x\right\rangle $ est pondérée par $\omega$
en tant que sous-algèbre de $A$, et d'après la proposition \ref{prop:UnicPond},
en tant qu'algèbre de rétrocroisement $K\left\langle x\right\rangle $
est aussi pondérée par $\omega$. 

D'après le lemme \ref{baseK<x>}, la famille $\left(p_{k}\right)_{k\geq1}$
où $p_{k}=x^{k}-x^{k-1}$ est génératrice de $K\left\langle x\right\rangle $,
qui est par hypothèse de dimension $n,$ donc $\left(p_{k}\right)_{1\leq k\leq n}$
est une base de $K\left\langle x\right\rangle $. Soit $T\left(X\right)=X^{n}-\sum_{i=1}^{n}\alpha_{i}X^{i-1}$
avec $\sum_{i=1}^{n}\alpha_{1}=1$, on a 
\[
T\left(X\right)=\left(X-1\right)\left(X^{n-1}-\sum_{k=1}^{n}\left(\sum_{i=k+1}^{n}\alpha_{i}-1\right)X^{k-1}\right).
\]
Ensuite, de $x^{k}=\sum_{i=1}^{k}p_{i}$ pour tout $1\leq k\leq n$
et de $T\left(L_{x}\right)x=0$ il résulte que $p_{n+1}=\sum_{k=2}^{n}\left(\sum_{i=k}^{n}\alpha_{i}-1\right)p_{k}$.
Soit $e=p_{1}+\sum_{i=2}^{n}\lambda_{i}p_{i}$, l'identité $e^{2}=e$
se traduit par le système d'équations d'inconnues $\lambda_{2},\ldots,\lambda_{n}$
:

\medskip{}
\[
\left\{ \begin{array}{rc}
-\lambda_{2}+2\left(\sum_{i=2}^{n}\alpha_{i}-1\right)\lambda_{n}=1 & \medskip\\
2\lambda_{k-1}-\lambda_{k}+2\left(\sum_{i=k}^{n}\alpha_{i}-1\right)\lambda_{n}=0 & \;\left(3\leq k\leq n\right)
\end{array}\right.
\]
dont le déterminant $\Delta$ est au facteur $2^{n-1}$ près la valeur
en $\frac{1}{2}$ du polynôme caractérisque de la matrice compagnon
de $\frac{T\left(X\right)}{X-1}$, on a donc $\Delta=-2^{n}T\left(\frac{1}{2}\right)\neq0$,
par conséquent le système d'équations a une solution et il existe
un idempotent dans $K\left\langle x\right\rangle $.
\end{proof}
On a vu dans \cite{M-V-02} que dans toute train algèbre ayant $\frac{1}{2}$
pour train racine, l'existence d'un idempotent n'est pas certaine. 

\medskip{}

La proposition ci-dessus est un cas particulier du théorème suivant
mais elle intervient dans la démonstration de celui-ci.
\begin{theorem}
Soit $A$ une algèbre pondérée vérifiant une identité non homogène
$f\in K\left\langle X\right\rangle $ de degré $\geq2$. S'il existe
$x\in H_{\omega}$ tel que la sous-algèbre $K\left\langle x\right\rangle $
soit de rétrocroisement de dimension $\deg\left(\vartheta_{f}\right)-1$
et si $\vartheta_{f}^{\prime}\left(\frac{1}{2}\right)\neq0$, alors
$A$ possède un idempotent. \end{theorem}
\begin{proof}
Soient $\left(A,\omega\right)$ une algèbre qui vérifie une identité
non homogène $f\in K\left\langle X\right\rangle $ de degré $n\geq2$
telle que $\vartheta_{f}^{\prime}\left(\frac{1}{2}\right)\neq0$.
Si $n=2$ il est immédiat que $A$ contient un idempotent, soit $n\geq3$
et supposons qu'il existe $x\in H_{\omega}$ tel que $K\left\langle x\right\rangle $
soit de rétrocroisement de dimension $\deg\left(\vartheta_{f}\right)-1$.
Par un raisonnement analogue à celui de la proposition précédente
on a $K\left\langle x\right\rangle $ pondérée par $\omega$. D'après
le théorème \ref{Mut&poids1}, la sous-algèbre $K\left\langle x\right\rangle $
est de mutation pondérée par $\omega$, il existe donc une application
linéaire $M:K\left\langle x\right\rangle \rightarrow K\left\langle x\right\rangle $
telle que $yy^{\prime}=\frac{1}{2}\left[\omega\left(y^{\prime}\right)M\left(y\right)+\omega\left(y\right)M\left(y^{\prime}\right)\right]$
pour tout $y,$ $y^{\prime}\in K\left\langle x\right\rangle $. Or,
d'après \cite{N-V-11a} il existe $D\in K\left[X\right]$ tel que
l'identité $f=0$ s'écrit $D\left(M\right)=0$ sur $H_{\omega}$ et
on a $\vartheta_{f}\left(X\right)=D\left(2X\right)$, d'où $\deg\left(D\right)=\deg\left(\vartheta_{f}\right)$.
Or pour tout entier $k\geq1$ et tout $y\in K\left\langle x\right\rangle \cap H_{\omega}$
on a $M^{k}\left(y\right)=2^{k-1}y^{k+1}-\sum_{i=2}^{k}2^{i-2}y^{i}$,
il en résulte que l'identité $D\left(M\right)=0$ s'écrit sur $K\left\langle x\right\rangle \cap H_{\omega}$
sous la forme d'une train identité aux puissances principales $g\left(y\right)=T\left(L_{y}\right)\left(y\right)=0$
avec $T\in K\left[X\right]$, il s'ensuit d'après \cite{N-V-11a}
que l'on a $D\left(2X\right)=\vartheta_{g}\left(X\right)$ par conséquent
$\vartheta_{g}\left(X\right)=\vartheta_{f}\left(X\right)$ et $\deg\left(T\right)=\deg\left(\vartheta_{f}\right)-1$,
or on a vu dans l'exemple (\ref{PolConjTrainPrinc}) que $\vartheta_{g}^{\prime}\left(\frac{1}{2}\right)=-4T\left(\frac{1}{2}\right)$
on a donc $T\left(\frac{1}{2}\right)=-\frac{1}{4}\vartheta_{f}^{\prime}\left(\frac{1}{2}\right)\neq0$
ce qui entraîne d'après la proposition précédente qu'il existe un
idempotent dans $K\left\langle x\right\rangle $. 
\end{proof}
Une algèbre $\left(A,\omega\right)$ est train plénière de rang $n\geq1$
(ou de degré $2^{n}$) si elle vérifie une identité du type $f\left(X\right)=X^{\left[n+1\right]}-\sum_{k=1}^{n}\alpha_{k}X^{\left[k\right]}$,
en prenant $P\in K\left[X\right]$ tel que $P\left(X\right)=X^{n}-\sum_{k=1}^{n}\alpha_{k}X^{k-1}$,
ceci est équivalent à $P\left(q\right)=0$ sur $H_{\omega}$, où $q:A\rightarrow A$,
$x\mapsto x^{2}$, ce polynôme $P$ est appelé le polynôme plénier
de $A$.
\begin{corollary}
Une algèbre $A$ train plénière de rang $n\geq2$, contient un idempotent
s'il existe $x\in H_{\omega}$ tel que la sous-algèbre $K\left\langle x\right\rangle $
soit de rétrocroisement de dimension $n-1$ et si $1$ est racine
simple du train polynôme plénier de $A$.\end{corollary}
\begin{proof}
Soient $\left(A,\omega\right)$ une algèbre vérifiant $f\left(X\right)=X^{\left[n+1\right]}-\sum_{k=1}^{n}\alpha_{k}X^{\left[k\right]}$
et $x\in H_{\omega}$ tel que $K\left\langle x\right\rangle $ soit
de rétrocroisement de dimension $n-1$. Dans l'algèbre de mutation
$\left(K\left\langle x\right\rangle ,M,\omega\right)$ on a $y^{\left[k\right]}=M^{k-1}\left(y\right)$
pour tout $y\in K\left\langle x\right\rangle \cap H_{\omega}$, il
en résulte que $f\left(y\right)=\left(M^{n}-\sum_{k=1}^{n}\alpha_{k}M^{k-1}\right)\left(y\right)$
d'où $\vartheta_{f}\left(X\right)=2^{n}X^{n}-\sum_{k=1}^{n}2^{k-1}\alpha_{k}X^{k-1}$
et $\vartheta_{f}^{\prime}\left(\frac{1}{2}\right)=2\left(n-\sum_{k=2}^{n}\left(k-1\right)\alpha_{k}\right)=2P'\left(1\right)$
où $P\left(X\right)=X^{n}-\sum_{k=1}^{n}\alpha_{k}X^{k-1}$. On a
donc $\vartheta_{f}^{\prime}\left(\frac{1}{2}\right)\neq0$ si et
seulement si $P'\left(1\right)\neq0$. Par conséquent, si $P'\left(1\right)\neq0$
d'après le théorème ci-dessus il existe un idempotent. 
\end{proof}
Si $P'\left(1\right)=0$ on a vu dans \cite{N-V-11a} (proposition
4) que dans ce cas l'existence d'un idempotent n'est pas certaine.
\begin{remark}
Dans\cite{B-H-10} le théorème 2 et dans\cite{LabSua-07} les propositions
4, 5, 8 sont des cas très particuliers et des conséquences immédiates
de ce résultat, car dans ces travaux les auteurs supposent que les
algèbres étudiées sont de rétrocroisement. \end{remark}

\end{document}